%% file: ShBlGrUp3.tex
\documentclass[11pt]{article}
\usepackage [english]{babel}
\usepackage{amsmath,amssymb,amsthm}
 \usepackage{graphicx}
\usepackage{mathrsfs}
\usepackage{xcolor}
\usepackage[affil-it]{authblk}
\newcounter{razdel}[section]

\newtheoremstyle{Mystyle}
     {\topsep}
     {\topsep}
     {\it}%         Body font
     {}%         Indent amount (empty = no indent, \parindent = para indent)
     {\bfseries}% Thm head font
     { }%        Punctuation after thm head
     { }%     Space after thm head (\newline = linebreak)
     {\thmnumber{#2.~}\thmname{#1}\thmnote{ #3}.}%         Thm head spec
\theoremstyle{Mystyle}
\newtheorem{Theorem}[razdel]{Theorem}
\newtheorem{Proposition}[razdel]{Proposition}
\newtheorem{Lemma}[razdel]{Lemma}

\newtheorem{Remark}[razdel]{Remark}

\newcommand \rem[1]{}

\begin{document}

\title{Shadowing for differential equations with grow-up.}

\author{Alexey V. Osipov%$^{\rm a}$$^{\rm b}$$^{\ast}$
\thanks{ Email: osipovav28@googlemail.com %\vspace{6pt}\\\vspace{9pt}  
}
}
\affil{Centro di ricerca matematica Ennio de Giorgi, Scuola Normale Superiore,
Piazza dei Cavalieri 3,
Pisa, 56100, Italy\\ Chebyshev Laboratory, Saint-Petersburg State University,
14th line of Vasilievsky island, 29B, Saint-Petersburg, 199178, Russia}

\maketitle

\abstract{We consider the problem of shadowing for differential equations with grow-up. We introduce so-called nonuniform shadowing properties (in which size of the error depends on the point of the phase space) and prove for them analogs of shadowing lemma. Besides, we prove a theorem about weighted shadowing for flows. We compactify the system (using Poincare compactification, for example), apply the results about nonuniform or weighted shadowing to the compactified system, and then transfer the results back to the initial system using the decompactification procedure.}

%\begin{keywords}
\textbf{Keywords. }
Shadowing, grow-up, hyperbolicity, Poincare compactification, time change.
%\end{keywords}

\section{Introduction and main definitions.}

Consider a system of ODEs 
\begin{equation}
\label{0}
\dot{x} = X(x),\quad x\in\mathbb{R}^N.
\end{equation}
 We say that %it has blow-up if it has a solution $|x(t,x_0)|\rightarrow\infty$ as $t\rightarrow t_0$ (see eg. $\dot{x} = x^2 + 1$), and we say that 
it has \textit{grow-up} if it has a solution $|x(t,x_0)|\rightarrow\infty$ as $t\rightarrow +\infty$. 

In the modern literature there are a lot of works devoted to study of grow-up and blow-up (a solution ''reaches'' infinity within a finite time) both for ODEs and PDEs (see, e.g., \cite{Ben, FM, VF}). Developing theory of shadowing for such equations seems to be an interesting and challenging problem.

Theory of shadowing studies the problem of closeness of approximate and exact trajectories of dynamical systems. Roughly speaking, a dynamical system has a shadowing property if any sufficiently precise approximate trajectory is close to some exact trajectory. 
We are interested in introducing shadowing properties for differential equations with grow-up and in obtaining relevant criteria. Thus we want to answer the following question (under reasonable assumptions): suppose we have a reasonable approximate solution going to infinity for infinite time; is it true that there exists an exact solution that is in some sense close to our approximate solution?

Usually theory of shadowing (see \cite{Palm, Pil} for review of classical results and \cite{Pil2} for review of modern results) establishes shadowing properties for dynamical systems on a compact phase space or establishes shadowing properties in a small neighborhood of a compact invariant set (e.g., shadowing near a hyperbolic set). Note that we deal with a dynamical system on a noncompact phase space (i.e. on $\mathbb{R}^N$).
%%%%%ÃÈÏÅÐÁÎËÈ×ÅÑÊÈÅ ÑÈÑÒÅÌÛ, ÏËÈÑÑ, ÃÈÏÅÐÁÎËÈ×ÅÑÊÀß ÑÒÐÓÊÒÓÐÀ ÍÀ ÂÑÅÌ ÔÀÇÎÂÎÌ ÏÐÎÑÒÐÀÍÑÒÂÅ

It is reasonable to act according to the following plan:
\begin{enumerate}
	\item to compactify our system (using, e.g., Poincare compactification),
	\item to establish some shadowing property for the compactified system,
	\item to transfer the property back to the original system.
\end{enumerate}

%Note that this plan is only for the case of grow-up: however one could reduce the case of blow-up to the case of grow-up via a time change.

It is relatively easy to understand that the standard shadowing property for flows (we will remind the definition below in the paper) is bad for this scheme. In order to act according to this scheme, one should consider shadowing properties with errors decreasing to zero sufficiently fast (weighted shadowing) or shadowing properties with errors depending on the point of the phase space (we call it nonuniform shadowing).
%%%%limit shadowing????

The rest of the paper is organized as follows: in Section 2 we discuss compactifications (Step 1 of the plan), %in section 3 we discuss polynomial gradient systems (the source of direct applications of our results), 
Section 3 is a brief introduction to classical theory of shadowing, in Section 4 we define and study nonuniform shadowing properties, in Section 5 we study weighted shadowing properties, and in Section~6 we discuss plans for further research.

Main results of the paper are Theorems \ref{Th3} and \ref{Th6} and their compactified versions Theorems \ref{ThM} and \ref{Th4}.

\section{Poincare compactification.}

It is possible to compactify our system $(\ref{0})$ in various ways (see \cite{Hell} for excellent description of compactifications). The most obvious way is just to add one point as infinity. If we do it, we will get a system (or vector field) on the $N$-dimensional sphere, $S_{N}$, without of one point. But, of course, in general, in order to get a vector field on $S_{N}$, (consider, e.g., any system with blow-up) we should apply a time change of a certain type. This procedure (the compactification of space and the time change) is called Bendixon compactification. It works not for an arbitrary vector field, but only for so-called normalizable vector fields. Any polynomial vector field belongs to the class of normalizable vector fields.

However we are not going to apply Bendixon compactification by the following reasons:
\begin{enumerate}
	\item it is very likely that the point on $S_N$ corresponding to infinity will be a degenerate point of very high order;
	\item Bendixon compactification does not allow to distinguish ''convergence to infinity by different directions''.
\end{enumerate}
Instead we will use the procedure called Poincare compactification. Similarly with Bendixon compactification it consists of two phases: a compactification of the phase space and a change of time.

We compactify the phase space in the following way: we consider the map $\Theta:\mathbb{R}^N\mapsto B_N$ defined by the formula
\begin{equation}
\label{CP}
%\Theta(x) := \left(\frac{x}{\sqrt{(|x|^2 + 1)}},\frac{1}{\sqrt{(|x|^2 + 1)}}\right),
\Theta(x) := \frac{x}{\sqrt{(|x|^2 + 1)}},
\end{equation}
where the coordinates in the $N$-dimensional ball, $B_N$, are chosen like on Fig.~\ref{Fig1}.

%\setlength\fboxsep{0pt}
%\setlength\fboxrule{0.5pt}
%\fbox{
\begin{figure}%[h]
	\centering
\def\svgwidth{\columnwidth}
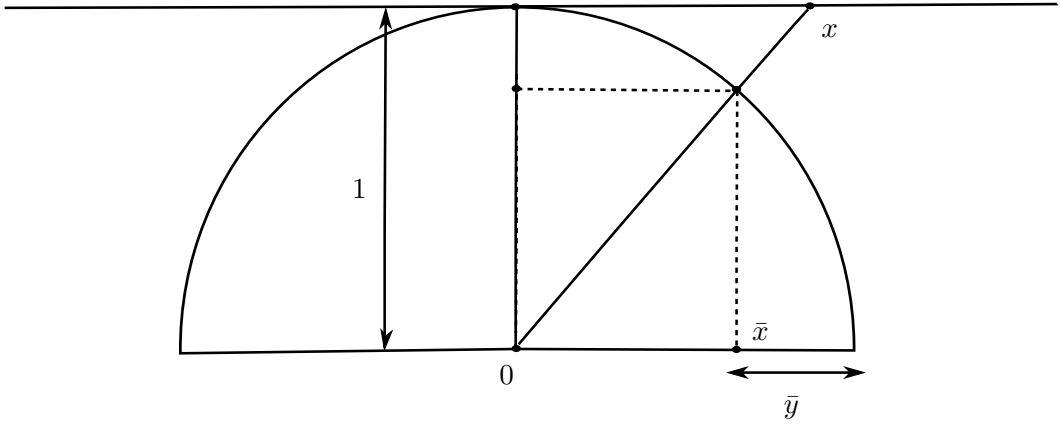
\caption{Compactification of the phase space.}
\label{Fig1}
\end{figure}

If we apply the compactification $(\ref{CP})$ to the system $(\ref{0})$, we will get a system $\dot{\bar{x}}=\bar{X}(\bar{x})$ on $B_N\backslash\partial B_N$ (hereinafter $\partial B_N$ denotes the boundary of the $N$-dimensional ball $B_N$). It is easy to understand that, in general, we will not get a system on $B_N$ (consider, e.g., any system with blow-up). Similarly with Bendixon compactification in order to get a system on $B_N$ we should apply a time change of a certain type. This procedure (the change of phase space and the change of time) is called \textit{Poincare compactification}.
Poincare compactification is defined not for an arbitrary vector field $(\ref{0})$, but only for the class of so-called normalizable vector fields. Any polynomial vector field belongs to the class of normalizable vector fields. Hereinafter we assume that we deal with normalizable vector fields.

%%%%Íàâåðíîå, ýòî íå íóæíî, ïîòîìó ÷òî èòàê ïîíÿòíî, ÷òî îøèáêè íå óâåëè÷àòñÿ, à ñæàòüñÿ â öåëîì îíè òîæå íå ìîãóò, à ñæàòèå ïî êîíêðåòíîìó íàïðàâëåíèþ íàì ìàëî ÷òî äàñò.

Consider polar coordinates $x=(z,\phi_1,\ldots,\phi_{N-1})$ in $\mathbb{R}^N$. Consider polar coordinates in $B_N$: $\bar{x}=\left(\bar{z}, \phi_1,\ldots,\phi_{N-1} \right)$. Naturally the compactification map~$(\ref{CP})$ can be rewritten in the following way:
\begin{equation}
\label{rR}
\bar{z} = \sqrt{\frac{z^2}{1 + z^2}} = \sqrt{1 - \frac{1}{z^2 + 1}}
\end{equation}
and angles $\phi_1,\ldots,\phi_{N-1}$ do not change.

 Consider a ball $U(\bar{R},\bar{x})\subset B_N$ of radius $\bar{R}$. We assume that $U(\bar{R},\bar{x})$ does not intersect the boundary. We want to find (a reasonably small) $R$ such that
$$\Theta^{-1}(U(\bar{R},\bar{x})) \subset U(R,x),$$
i.e. we want to understand how the ball $U(\bar{R},\bar{x})$ can be expanded via the decompactification procedure.

Note that, since we are interested only in getting a qualitative estimate, and polar coordinates and Cartesian coordinates generate equivalent topologies, it is enough to consider only the change of radial coordinates.
Put $\bar{y} = 1-\bar{z}$ (i.e. $\bar{y}$ is radial distance to the boundary). Then it easy to compute that
$$z = \sqrt{\frac{1}{2\bar{y} - \bar{y}^2} - 1}.$$
 Assuming that the ball does not intersect the boundary, the points $(\bar{y}-\bar{R},\ldots)$ and $(\bar{y} + \bar{R},\ldots)$ are mapped to the points
$(\sqrt{-1 + 1/(2(\bar{y}-\bar{R}) - (\bar{y}-\bar{R})^2)},\ldots)$ and $(\sqrt{-1 + 1/(2(\bar{y}+\bar{R}) - (\bar{y}+\bar{R})^2)},\ldots)$. After careful  calculations we see that
$$
\Theta^{-1}(U(\bar{R},\bar{x})) \subset U(R,\Theta^{-1}(x)),\quad
\mbox{where }R = O(\bar{R}/(\bar{y}^{3/2})).
$$
and taking into consideration that
\begin{equation}
\label{transferprime} 
%\bar{x} = 1 - \sqrt{1 - \frac{1}{z^2}} = O(\frac{1}{z^2}) = O(\frac{1}{|y|^2}),
\bar{y} = 1 - \sqrt{1 - \frac{1}{z^2 + 1}} = O\left(\frac{1}{z^2}\right) = O\left(\frac{1}{|x|^2}\right).
\end{equation}
we get
\begin{equation}
\label{transfer}
\Theta^{-1}(U(\bar{R},\bar{x})) \subset U(R,\Theta^{-1}(x)),\quad
\mbox{where }R = O(\bar{R}|x|^{3}).
\end{equation}

Now we consider the inverse problem. Consider a ball $U(R,x)\subset\mathbb{R}^N$. We want to find (a reasonably small) $\bar{R}$ such that
$$\Theta(U(R,x))\subset U(\bar{R},\bar{x}).$$
By \eqref{rR}, the points $(z-R,\ldots)$ and $(z+R,\ldots)$ are mapped to the points
$(1-\sqrt{1 - 1/((z-R)^2 + 1)},\ldots)$ and $(1-\sqrt{1-1/((z+R)^2+1)},\ldots)$. After careful calculations and using \eqref{transferprime}, we observe that
\begin{equation}
\label{transferalt}
\Theta(U(R,x))\subset U(\bar{R},\Theta(x)),\qquad\mbox{where }\bar{R}=O(R/|x|^3)=O(R|\bar{y}|^{3/2}).
\end{equation}

\section{Standard shadowing properties.}

Consider a diffeomorphism $f$ of a compact smooth Riemannian manifold $M$ with Riemannian metric dist. A \textit{trajectory} of a point $q$ of the diffeomorphism $f$ is the sequence 
$$O(q,f) = \{f^k(q)\}_{k\in\mathbb{Z}}.$$
A sequence $\{x_k\}_{k\in\mathbb{Z}}$ of points of $M$ is a $d$\textit{-pseudotrajectory} if
\begin{equation}
\label{pseudo}
\mbox{dist}(x_{k+1},f(x_k))\leq d\quad\forall k\in\mathbb{Z}.
\end{equation}
Clearly the notion of a pseudotrajectory is one of possible formalizations of the notion of an approximate trajectory.

A diffeomorphism $f$ has \textit{standard shadowing property} if for any $\epsilon>0$ there exists $d_0$ such that for any $d$-pseudotrajectory $\{x_k\}_{k\in\mathbb{Z}}$ with $d\leq d_0$ there exists a point $q$ such that
$$\mbox{dist}(x_k,f^k(q))\leq\epsilon\quad\forall k\in\mathbb{Z}.$$
Thus standard shadowing means that any sufficiently precise pseudotrajectory is pointwisely close to some exact trajectory.

This property is also called \textit{two-sided} standard shadowing property, because biinfinite trajectories and pseudotrajectories are considered. Also so-called \textit{one-sided} standard shadowing property is considered, in which pseudotrajectories and trajectories are indexed by natural numbers (clearly this property is weaker for diffeomorphisms than the two-sided version).  Moreover, so-called \textit{Lipschitz} standard shadowing property is considered (in which $d=\epsilon/L$, where $L$ is a global constant).

One of the main results of theory of shadowing is so-called shadowing lemma (see \cite{An,Bow}):
\begin{Theorem}[(Anosov, Bowen)]
A diffeomorphism has Lipshitz standard shadowing property in a small neighborhood of a hyperbolic set.
\end{Theorem}
Recently the following result was obtained (see \cite{PT}):
\begin{Theorem}[(Pilyugin, Tikhomirov)]
Lipschitz standard shadowing property is equivalent to structural stability.
\end{Theorem}

For flows the situation with shadowing properties is more difficult. First of all, there is no canonical way to formalize the notion of a pseudotrajectory for a flow. We will use here the definitions offered by S.Yu. Pilyugin (see \cite{Pil}).

Let $\Phi$ be a flow on a compact smooth Riemannian manifold $M$. A $(d,T)$-\textit{psedotrajectory} of a flow $\Phi$ is a function $\Psi:M\mapsto\mathbb{R}$ such that
$$\mbox{dist}(\Psi(t+\tau),\Phi(\tau,\Psi(t)))\leq d\quad\forall |\tau|\leq T,\forall t\in\mathbb{R}.$$
Note that a function $\Psi$ is not assumed to be continuous.

Let $Rep$ be the class of all increasing homeomorphisms of $\mathbb{R}$. Put
$$Rep(\epsilon)=\left\{\alpha\in Rep\mid \left|\frac{\alpha(t) - \alpha(s)}{t-s}\right|\leq\epsilon\quad\forall t\neq s\right\}.$$

A flow $\Phi$ has \textit{oriented shadowing property} if for any $\epsilon>0$ there exists $d_0$ such that for any $(d,1)$-pseudotrajectory with $d\leq d_0$ there exist a point $q$ and a reparametrization $\alpha\in Rep$ such that
$$\mbox{dist}(\Psi(t),\Phi(\alpha(t), q))\leq\epsilon\quad\forall t\in\mathbb{R}.$$
%Note that oriented shadowing property is preserved via time changes.

It is necessary to use time reparametrizations because of possible existence of periodic trajectories. However if a flow has no periodic trajectories, but is good (e.g., is a Smale flow), then no reparametrizations are required.

A flow $\Phi$ has \textit{standard shadowing property} if for any $\epsilon>0$ there exists $d_0$ such that for any $(d,1)$-pseudotrajectory with $d\leq d_0$ there exist a point $q$ and a reparametrization $\alpha\in Rep(\epsilon)$ such that
$$\mbox{dist}(\Psi(t),\Phi(\alpha(t), q))\leq\epsilon\quad\forall t\in\mathbb{R}.$$
Standard shadowing property is not preserved via time changes. Similarly with the case of discrete time systems, Lipschitz version of standard shadowing property can be defined (when $d=\epsilon/L$, where $L$ is a global constant).

Shadowing lemma for flows was proved by S.Yu. Pilyugin and K. Palmer:
\begin{Theorem}[(Pilyugin)]
\label{Thfl}
A flow has Lipschitz shadowing property in a small neighborhood of a hyperbolic set.
\end{Theorem}

Palmer, Pilyugin, and Tikhomirov obtained the following result (see \cite{PPT}):
\begin{Theorem}[(Palmer, Pilyugin, Tikhomirov)]
Structural stability for flows is equivalent to Lipschitz shadowing property.
\end{Theorem}

\section{Nonuniform shadowing.}

\subsection{Definitions and basic results.}

Let $M$ be a smooth compact $N$-dimensional Riemannian manifold with boundary $\partial M$. 
By Whitney theorem, we assume that $M$ is embedded into an Euclidean space of a sufficiently large dimension.
For any $x\in M$ define 
$$r(x) = \mbox{dist}(x,\partial M) = \min_{y\in\partial M}|x-y|.$$ 
Without of loss of generality, we assume that $M$ has diameter less than $1$. 
Consequently, $r(M)\subset [0,1]$.
Denote $\mbox{Int}(M) = M\backslash\partial M$.

A sequence $\{x_k\}_k\subset\mbox{Int}(M)$ is a \textit{nonuniform} $(n,\delta)$\textit{-pseudotrajectory} ($n\geq 1$) if 
\begin{equation}
\label{dprop}
|x_{k+1} - f(x_k)|\leq d(r(f(x_k))),\quad\forall k\geq 0,
\end{equation}
where 
\begin{equation}
\label{ddef}
d(z) = \delta z^n,\quad\forall z\in\mathbb{R}_{>0}.
\end{equation}
\begin{Remark}

1) Any $(n,\delta)$-pseudotrajectory is a $\delta$-pseudotrajectory (in the classical sense).

2) If we put $d(r(x_{k+1}))$ in \eqref{dprop} in the definition of an $(n,\delta)$-pseudotrajectory, then %the definition does not remain the same. However, the previous definition seems more natural to us.
we obtain an equivalent definition. Note that $n\geq 1$.
\end{Remark}

We say that a diffeomorphism $f$ of $M$ has \textit{nonuniform shadowing property} with exponent $m\geq 0$ if
for any number $\Delta$ and the function 
\begin{equation}
\label{epsdef}
\epsilon(z) = \Delta z^m,\quad\forall z\in\mathbb{R}
\end{equation} 
there exist numbers $\delta_0$ and $n_0$ 
%there exists a number $n_0$ such that for any $n\geq n_0$ there exists a number $\delta_0(n_0)$ such that for any $\delta\leq \delta(n_0)$ 
such that for any nonuniform $(n,\delta)$-pseudotrajectory with $\delta\leq\delta_0$ and $n\geq n_0$ 
there exists a point $q$ such that
\begin{equation}
\label{epsprop}
|x_k - f^k(q)|\leq\epsilon(r(f^k(q))),\quad\forall k\geq 0.
\end{equation}
\begin{Remark}

1) For $m=0$ and $n_0=0$ this property is standard shadowing property.

2) It is possible to put $\epsilon(r(x_k))$ instead of $\epsilon(r(f^k(q)))$ in \eqref{epsprop} in the previous definition, but it does not lead to an equivalent definition, generally speaking. However the previous definition seems more natural to us, and the definition remains equivalent if $m\geq 1$. 
%but this will lead to an equivalent definition.
\end{Remark}

For flows on $M$ this concept can be defined in the following way.
A (not necessarily continuous) function $\Psi:\mathbb{R}\mapsto \mbox{Int}(M)$ is a \textit{nonuniform} $(n,\delta,T)$\textit{-pseudotrajectory} if
for the function
$$\psi(t) := \max_{|\tau|\leq T}|\Psi(t+\tau) - \Phi(\tau,\Psi(t))|,$$
and the function
$d(\cdot)$ defined by \eqref{ddef}
the following holds:
\begin{equation}
\label{nudpstfl}
\psi(t)\leq \min_{|\tau|\leq T}d(r(\Phi(\tau,\Psi(t))))\quad\forall t\in\mathbb{R}_{>0}.
\end{equation}

A flow $\Phi$ has \textit{nonuniform shadowing property} with exponent $m$ if for any number $\Delta$ and the function $\epsilon(\cdot)$ defined by \eqref{epsdef} 
there exist numbers $\delta_0$ and $n_0$ such that for 
%there exists a number $n_0$ such that for any $n\geq n_0$ there exists a number $\delta_0=\delta_0(n)$ such that for any $\delta\leq\delta_0$ and 
any nonuniform $(n,\delta,1)$-pseudotrajectory with $\delta\leq\delta_0$ and $n\geq n_0$ 
there exists a point $q$ such that
$$|\Psi(t) - \Phi(t,q)|\leq \max_{[t]\leq \tau\leq [t]+1}\epsilon(r(\Phi(\tau,q))),\quad\forall t\geq0,$$
where $[t]$ is the maximal integer number no more than $t$.

Note that oriented shadowing property is a nonuniform oriented shadowing property with exponent $0$.
We do not use reparametrizations in this property since in (very specific) situations that we will consider it is possible to choose the identity map as the reparametrization.

We use the following proposition:
\begin{Proposition}
\label{PropA}
Consider the time-one map $f$ for a flow $\Phi$. Suppose that $f$ has a nonuniform shadowing property with exponent $m$; then $\Phi$ has nonuniform shadowing property with exponent $m$.
\end{Proposition}

\begin{proof}
Let $\Psi$ be a nonuniform $(\delta,n,1)$-pseudotrajectory. Consider the sequence 
$\xi=\{x_k\}_{k\in\mathbb{Z}} = \{\Psi(k)\}_{k\in\mathbb{Z}}$. We claim that $\xi$ is a $(\delta, n)$-pseudotrajectory. Indeed, 
$$|x_{k+1}-f(x_k)|\leq \psi(k)\leq d(r(f(x_k))).$$
Choose a point $q$ such that \eqref{epsprop} holds. Fix any $t\in [k,k+1]$. Define
$$H=\max_{x\in M, 0\leq \tau\leq 1} \left|\left|\frac{\partial\Phi}{\partial x}(\tau, x)\right|\right|.$$
Then, by \eqref{nudpstfl},
$$|\Psi(t) - \Phi(t,q)|\leq |\Psi(t) - \Phi(t-k,\Psi(k))| + |\Phi(t-k,\Psi(k)) - \Phi(t,q)|
\leq$$ 
$$\leq d(r(\Phi(t-k,\Psi(k)))) + H|\Psi(k) - \Phi(k,q)| \leq  d(r(\Phi(t-k,\Psi(k)))) + $$ 
$$ + H\epsilon(r(\Phi(k,q)))\leq (1+H)\max_{k\leq\tau\leq k+1}\epsilon(r(\Phi(\tau,q))).$$
\end{proof}

Now let us investigate how nonuniform shadowing property is preserved via the decompactification procedure.

\begin{Proposition}
\label{PropB}
Suppose that the compactified flow has nonuniform shadowing property with exponent $\bar{m}$ and numbers $\bar{\delta}$ and $\bar{n}_0$.

Then the initial flow has the following analog of nonuniform shadowing property, which we call \textit{noncompact nonuniform oriented shadowing property}:

There exists a time change $\alpha:\mathbb{R}\times \mathbb{R}^N\mapsto\mathbb{R}$ such that for any function 
$\epsilon(t) = \Delta |t|^{-2\bar{m}+3}$ there exist numbers $\delta_0$ and $n_0=2\bar{n}_0-3$ such that %for any $n\geq n_0$ there exists $\delta_0=\delta_0(n)$ such that for any $\delta\leq\delta_0$, any function
for any function
$d(t)=\delta|t|^{-n}$ with $\delta\leq\delta_0$ and $n\geq n_0$ 
and any function $\Psi$ (which we call a \textit{noncompact nonuniform $(\delta,n,1)$-pseudotrajectory}) such that
$$\max_{|\tau|\leq 1}|\Psi(t+\tau) - \Phi(\alpha(\tau,\Psi(t)),\Psi(t))|\leq \min_{|\tau|\leq 1}d(|\Phi(\alpha(\tau,\Psi(t)),\Psi(t))|)$$
there exists a point $q\in\mathbb{R}^{N}$ and a reparametrization $\alpha\in Rep$ 
such that
$$|\Psi(t) - \Phi(\alpha(t,q),q)|\leq \max_{\alpha([t],q)\leq \tau\leq \alpha([t]+1,q)}\epsilon(|\Phi(\tau,q)|).$$
%$$|\Psi(t) - \Phi(\alpha(t,q),q)|\leq \epsilon(|\Phi(\alpha(t,q),q)|).$$
%%%% -2\bar{n}+3=n; \bar{n}=(3-n)/2\leq \bar{n_0}
\end{Proposition}
%$m$ is changed to $2m-3$, $n_0$ to $2n_0$ and $r(\cdot)$ to $|\cdot|^{-1}$.
%with the exponent $\ell=m-2$:
%instead of the function $r(\cdot)$ we consider $1/|\cdot|$ (due to $(\ref{transfer0})$).
%%%%%We need some explanation here

%(This is because in the compactification the distance to the border is $$z^{comp} = 1 - \arctan(|x^{noncomp}|)/(\pi/2))\approx C/|x^{noncomp}|.$$
%%%%%to get that \approx is sufficient to estimate the derivative
%%%%%1- \arctan = int_{>X}O(1/x^2) = O(1/X)

In particular, if $\bar{m}\geq 3/2$, then the initial noncompactified flow has oriented shadowing property. Even for $\bar{m}<3/2$ we still have some sort of shadowing (despite our errors grow as the pseudotrajectory goes to infinity). Thus these shadowing properties even for $\bar{m}<3/2$ can be used to determine grow-up.

\begin{proof}
Let $\alpha$ be the inverse map to the time change used in the compactification procedure.
Note that, by \eqref{transferprime} and \eqref{transferalt},
$$|\Theta(\Psi(t+\tau)) - \bar{\Phi}(\tau,\Theta(\Psi(t)))|\leq
|\Psi(t+\tau) - \Phi(\alpha(\tau,\Psi(t)),\Psi(t))||r(\bar{\Phi}(\tau,\Theta(\Psi(t))))|^{3/2}\leq$$
$$\leq \delta|\Phi(t,q)|^{-n}|r(\bar{\Phi}(\tau,\Theta(\Psi(t))))|^{3/2}\leq
\delta |r(\bar{\Phi}(\tau,\Theta(\Psi(t))))|^{n/2+3/2}$$
(strictly speaking, since we have used asymptotic inequalities, we should have written a multiplicative constant $C_0$ in the right side of the previous equation; however, for simplicity, in such cases we omit such multiplicative constants).

%%%%\bar{n_0} = 2n_0 + 3/2; n_0=\bar{n_0}-3/4.
Since $n/2+3/2\geq\bar{n}_0$, by assumption of the proposition, there exists a point $\Theta(q)$ such that
$$|\Theta(\Psi(t)) - \bar{\Phi}(t,\Theta(q))|\leq 
\Delta|r(\bar{\Phi}(t,\Theta(q)))|^{\bar{m}}.$$
Note that, by \eqref{transferprime} and \eqref{transfer},
$$|\Psi(t) - \Phi(\alpha(t,q),q)|\leq |\Theta(\Psi(t)) - \bar{\Phi}(\alpha(t,q),\Theta(q))||r(\bar{\Phi}(t,\Theta(q)))|^{-3/2}\leq$$
$$\leq \Delta\max_{[t]\leq \tau\leq [t]+1}|r(\bar{\Phi}(\tau,\Theta(q)))|^{\bar{m}-3/2}\leq
\Delta\max_{\alpha([t],q)\leq \tau\leq \alpha([t]+1,q)}|\Phi(\tau,q)|^{-2\bar{m}+3}.$$
\end{proof}

\subsection{Reasoning of Conley.}

Let $\{g_k:\mathbb{R}^N\mapsto \mathbb{R}^N\}$ be a sequence of diffeomorphisms.
It is possible to use a manifold instead of $\mathbb{R}^{N}$, since the reasoning is local.
 %We assume that there exists a constant $C$ such that
%$$||Dg_k||\leq C,\quad ||Dg^{-1}_k||\leq C.$$
We assume that this sequence of maps is hyperbolic on some compact locally maximal invariant set $\Lambda\subset \mathbb{R}^N$, and they locally preserve the foliation of the phase space on $\textbf{s}$-dimensional stable and $\textbf{u}$-dimensional unstable manifolds. 

%%%%%%%%%%%%%Lyaponov metric
Let $U(\epsilon_0,\Lambda)$ be a small neighborhood of $\Lambda$. 
%Consider a decreasing sequence of numbers $\{\delta_k\}_{k\geq 0}$.
Consider a continuous function $\delta(p)$ (later we impose additional restrictions on it).

Let $W_{2\delta(p)}^{s}(p)$ and $W_{2\delta(p)}^u(p)$ be $\textbf{s}$-dimensional and $\textbf{u}$-dimensional submanifolds in $\mathbb{R}^{N}$ of size $\delta(p)$ respectively (corresponding to stable and unstable manifolds at $p$).  
Consider the set $W_{2\delta(p)}(p)$, the neighborhood of the point $p$, and a map $\chi_{p}: W_{2\delta(p)}(p)\mapsto E_{2\delta(p)}=E_{2\delta(p)}^s\times E_{2\delta(p)}^u$, where $E_{2\delta(p)}$ is the standard cube, and stable manifolds are mapped to $s$-components and unstable manifolds are mapped to $u$-components.
Denote by $pr_s$ and $pr_u$ natural projections on $E_{2\delta(p)}^s$ along $E_{2\delta(p)}^{u}$ and on $E_{2\delta(p)}^{u}$ along $E_{2\delta(p)}^{s}$ respectively.
Fix $\{x_k\}_k\subset \mathbb{R}^N$ (in the applications $\{x_k\}_k$ will be a nonuniform pseudotrajectory).
Consider a sequence of neighborhoods $\{U(\delta(x_k),x_k)\}_{k}$.
Assume this neighborhoods are so small that 
\begin{equation}
\label{Udef}
U(\delta(x_k),x_{k})\subset W_{2\delta(x_{k})}(x_{k}),\quad
g_k(U(\delta(x_k),x_k))\subset W_{2\delta(x_{k+1})}(x_{k+1})\quad\forall k\geq0,
\end{equation} 
\begin{equation}
\label{stable}
pr_{s}\chi_{x_{k+1}}g_k(U(\delta(x_{k}),x_k)) \subset pr_s \chi_{x_{k+1}}U(\delta(x_{k+1}),x_{k+1})\quad\forall k\geq0,
\end{equation}
\begin{equation}
\label{unstable}
pr_{u}\chi_{x_{k+1}}g_k(U(\delta(x_k),x_{k})) \supset pr_u \chi_{x_{k+1}}U(\delta(x_{k+1}),x_{k+1})\quad\forall k\geq0.
\end{equation}
Clearly estimates $(\ref{stable})$ and $(\ref{unstable})$ hold, since $\delta(x_k)$ can be chosen (uniformly) sufficiently small, $\Lambda$ is a hyperbolic~set, and it is possible to consider as a new sequence of $\{g_k\}$ finite compositions $\left\{g_{T-1}\circ\ldots\circ g_0, g_{2T-1}\circ\ldots\circ g_{T},\cdots\right\}$, where $T$ is a sufficiently large number. Note that passage to such finite compositions preserves shadowing. 

\begin{figure}[h]
	\centering
\def\svgwidth{\columnwidth}
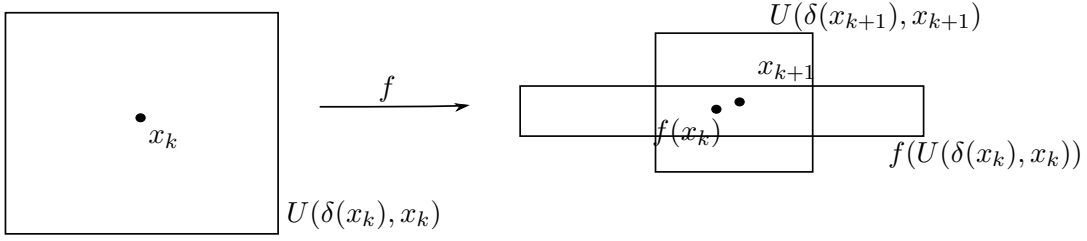
\caption{Mappings of squares.}
\label{Fig3}
\end{figure}

The dynamics of $\{g_k\}$ is depicted on Fig. \ref{Fig3} (vertical direction corresponds to contraction, and horizontal direction corresponds to expansion). 

Consider the map 
\begin{equation}
\label{hkdef}
h_{k} = \chi_{x_{k+1}}\circ g_k\circ \chi_{x_k}^{-1}.
\end{equation}
 Clearly the analogs of estimates $(\ref{stable})$ and $(\ref{unstable})$ hold for the maps $h_{k}$.

Consider the cube $E_{\delta(p)}$.
A \textit{horizontal} $\textbf{u}$\textit{-dimensional surface} is a surface $S\subset E_{\delta_0}$ such that $pr_{E^u}S = E^u$. A \textit{vertical }$\textbf{s}$\textit{-dimensional surface} is a surface $S\subset E_{\delta(p)}$ such that $pr_{E^s}S = E^{s}$.

We will need the following lemma. In essence, it was used without of proof in paper \cite{Con}. Paper \cite{PS} contains the proof of this lemma for two-dimensional case. In essence, we generalize this proof to the case of higher dimensions. 

\begin{Lemma}
\label{LK}
Let $g_k: M\mapsto M$ be a sequence of smooth maps,  
let $\{x_k\}_{k}$ and $\{\delta(x_k)\}_{k}$ be
such that relations $\eqref{Udef}$,
$(\ref{stable})$ and 
$(\ref{unstable})$ hold. 
Put $$\mbox{Inv}_+(\{x_k\}_{k\geq0},\{g_k\}):= \left\{ q\mid g_{k}\circ\ldots\circ g_0(q)\in U(\delta(x_{k+1}), x_{k+1}) \quad\forall k\geq 0\right\}.$$
Then the set $\chi_{x_0}\mbox{Inv}_+(\{x_k\}_{k\geq0},\{g_k\})$ contains a unique vertical $\textbf{s}$-dimensional surface.
\end{Lemma}

\begin{proof}
1) Consider the maps $h_{k}: E_{\delta(x_k)}\mapsto E_{\delta(x_{k+1})}$ defined by \eqref{hkdef}. Note that $\mbox{Inv}_+(\{0\}_{k\geq0},\{h_{k}\}_{k\geq0})=\chi_{x_0}\mbox{Inv}_+(\{x_k\}_{k\geq0},\{g_k\}_{k\geq0})$,
where $$\mbox{Inv}_+(\{0\}_{k\geq0},\{h_{k}\}_{k\geq0}) = \{q\mid h_{k}\circ\ldots\circ h_{0}(q)\in E_{\delta(x_{k+1})}(x_{k+1})\}.$$ 
In order to prove Lemma \ref{LK}, it is sufficient to prove that the set 
$\mbox{Inv}_+(\{0\}_{k\geq0},\{h_{k}\}_{k\geq0})$ contains a vertical $\textbf{s}$-dimensional surface.

2) The case of $\textbf{s}=N$ is trivial; that is why we do not consider it in details. In this case 
$\mbox{Inv}_+(\{0\}_{k\geq0},\{h_{k}\}_{k\geq0}) = E_{\delta(x_0)}(x_0)$.

3) We divide $E_{\delta(x_0)}(x_0)$ on cubes of size $1/2^n$. Denote by $\mbox{Inv}_n$ the set of all cubes that 
intersect $\mbox{Inv}_+(\{0\}_{k\geq0},\{h_{k}\}_{k\geq0})$.
By \eqref{unstable}, it is sufficient to prove that for any $n$ the set $\mbox{Inv}_n$ contains a vertical $\textbf{s}$-dimensional surface (since a limit of vertical $\textbf{s}$-dimensional surfaces as $n\rightarrow+\infty$ is a vertical $\textbf{s}$-dimensional surface).

4) Consider the set $B = E^s_{\delta(x_0)}(x_0)\times\partial E^u_{\delta(x_0)}(x_0)$. 
Let $J_n$ be the set of all cubes $C$ 
such that the following holds:
\begin{itemize}
	\item any cube $C\subset J_n$ can be connected with $B$ by a horizontal $\textbf{u}$-dimensional surface contained in $J_n$,
	\item any cube $C\subset J_n$ does not intersect $\mbox{Inv}_+(\{0\}_{k\geq0},\{h_{k}\}_{k\geq0})$. 
\end{itemize}
Due to $(\ref{stable})$, $J_n$ is not empty if $n$ is sufficiently large (and contains all cubes adjacent to $B$). 

5) Note that the set $\partial J_n$ contains $E^s_{\delta(x_0)}(x_0)\times\partial E^u_{\delta(x_0)}(x_0)$. Moreover, $J_n$ is homeomorphic to the direct product of $E^s_{\delta(x_0)}(x_0)\times\partial E^u_{\delta(x_0)}(x_0)$ and some other set. Consequently, $B$ contains a vertical $\textbf{s}$-dimensional surface. Its uniqueness follows from \eqref{unstable}. We have proved our lemma.
\end{proof}

%\begin{Remark}
%In the conditions of Lemma \ref{LK}, due to (\ref{stable}), $\chi_{x_k}\mbox{Inv}_+(\{x_k\}_{k\geq0},g_k)$ is a vertical $\textbf{s}$-dimensional surface.
%2) Naturally the analog of Lemma \ref{LK} holds
%\end{Remark}

\subsection{Shadowing lemma for nonuniform shadowing, case of diffeomorphisms.}

Let $f$ be a diffeomorphism of an $N$-dimensional manifold $M$ with the boundary~$\partial M$.
Note that, since we are interested in criteria for shadowing, it is possible to consider $f^{T}$ instead of $f$, where $T$ is a sufficiently large number. It allows to simplify hyperbolicity estimates.

\textbf{Setting.}  %Assume that there exists a hyperbolic set $\Lambda\subset\partial M$ for $f$ with the hyperbolicity constant $\lambda$. Also suppose that the following holds:
Assume that $f$ has a locally maximal compact invariant set $\Lambda\subset\partial M$ and the following holds (we assume either 4.a) or 4.b)):

\textbf{Main assumption.} For any point $p\in\Lambda$ there exists a one-dimensional subspace $\ell(p)\subset T_pM$ such that: 

0) $\ell(p)\notin T_{p}\partial M$;

1) $\ell(f(p)) = Df(p)\ell(p)$; 

2) $\ell(p)$ continuously depends on $p$; 

3) for any $v\in\ell(p)$ $\mu_1|v|\leq |Df(p)v|\leq\mu_2|v|$, and either $\mu_2<1$ or $\mu_1>1$ (hereinafter we consider only the case $\mu_2<1$, since the other case is completely similar, and we just get shadowing for negative indices instead of positive indices)

4.a) (if $\mu_2<1$) we choose $\lambda^s_{min}$ such that 
$|Df(p)v|\geq \lambda^{s}_{min}|v|$ for any $v\in\nobreak T_{p}M$

OR

4.b) the set $\Lambda$ is hyperbolic for $f$, hyperbolicity is controlled by constants
$\lambda^s_{min}<\lambda^s_{max}<1$ for the case of the stable space and $1<\lambda^u_{min}<\lambda^u_{max}$ for the case of the unstable space. 
 %The hyperbolicity on the border is controled by constants $\lambda^s_{min}<\lambda^s_{max}<1$ for the case of the stable space and $1<\lambda^u_{min}<\lambda^u_{max}$ for the case of the unstable space.

\begin{Theorem} 
\label{ThM}
Suppose that Main assumption holds. Let $U$ be a sufficiently small neighborhood of $\Lambda$ such that the analogues of estimates from Main Assumption hold in it.

1) Assume that $\mu_2<1$ and Item 4.a holds. %, i.e. we have an attraction towards the border. 
Suppose that
\begin{equation}
\label{muineq}
\mu_2^m < \lambda^s_{min},\quad m>\ln \lambda^s_{min}/\ln\mu_2,
\end{equation}
then $f$ has one-sided nonuniform shadowing property in $U$ with the exponent $m$ (which is Lipshitz if $m\geq 1$). If a pseudotrajectory is fully contained in $U$, then the point $q$ from the definition of nonuniform shadowing property is unique. If only a finite part of a pseudotrajectory is contained in $U$, then the set of points~$q$ such that the analog of $\eqref{epsprop}$ holds is a small ball.

2) Suppose that $\mu_2<1$, Item 4.b) holds, and
\begin{equation}
\label{muineq2}
\mu_1^m > \lambda^{s}_{max},\quad m<\ln \lambda^{s}_{max}/\ln\mu_1,
\end{equation}
then $f$ has  one-sided nonuniform shadowing property in $U$ with the exponent~$m$ (which is Lipschitz if $m\geq 1$). %Moreover the set of points $q$ such that the analog of $(\ref{epsprop})$ holds is a $s$-dimensional disk $D_s$.
 If a pseudotrajectory is fully contained in $U$, then the set of points $q$ such that the analog of $(\ref{epsprop})$ holds is an $s$-dimensional disk $D_s$. If only finite part of a pseudotrajectory is contained in $U$, then the set of points~$q$ such that the analog of $(\ref{epsprop})$ holds is a small neighborhood of an $s$-dimensional disk $D_s$.
 
%2.1) Assume that $\mu_1>1$, i.e. we have a repulsion towards the boundary. 
%Suppose that Item 4.a) holds and
%\begin{equation}
%\label{muineq3}
%$$\mu_1^m > \lambda^{u}_{max},\quad m>\ln \lambda^{u}_{max}/\ln\mu_1,$$
%\end{equation}
%Then $f$ has one-sided Lipshitz nonuniform shadowing property in $U$ with the exponent $m$ (which is Lipschitz if $m\geq 1$). Moreover the set of points $q$ such that the analog of $(\ref{epsprop})$ holds is a 
%small neighborhood of an $s$-disk.

%2.2) Assume that $\mu_1>1$ and Item 4.b) holds.%, i.e. we have a repulsion towards the border.
%Suppose that
%\begin{equation}
%\label{muineq4}
%$$\mu_2^m < \lambda^{u}_{min},\quad m<\ln \lambda^{u}_{min}/\ln\mu_2,$$
%\end{equation}
%where $s^+_{max}$ is the right singular value of $A$.
%Then $f$ has one-sided Lipshitz nonuniform shadowing property in $U$ with the exponent $m$ (which is Lipschitz if $m\geq 1$). Moreover the set of points $q$ such that the analog of $(\ref{epsprop})$ holds is a 
%small ball (small neighborhood of a $(s+u)$-ball).

\end{Theorem}

\begin{Remark}
%1) It is possible to give another formulation of the theorem if we apply Holder linearizations at $p$.

1) The simplest application of the theorem is when $\Lambda$ is a hyperbolic fixed point. In this case Main Assumption holds and the conditions of Theorem are naturally formulated in terms of eigenvalues of the corresponding matrix.% In this case Main assumption holds. For convenience, we give a precise formulation of this case.
%Suppose that $\mu<1$ is the eigenvalue that corresponds to the direction towards the border.
%Let $\lambda^{s}_{min}$ and $\lambda^{s}_{max}$ be the minimal and the maximal stable eigenvalues.
%Then, since $\lambda^{s}_{min}\leq \mu$, condition \eqref{muineq} is possible only if $m>1$.
%Similarly, since 

2) Note that condition \eqref{muineq} implies that $m>1$, and condition \eqref{muineq2} implies that $0<m<1$.

3) %One of the applications is the case when the theorem about filtrations holds.
It is possible to give a more refined, stronger, and more technical version of the theorem using the theorem about filtrations (see~\cite{Pes}). 

\end{Remark}

\begin{proof} 

\textbf{Step 1.} Introduction of the coordinates. Here we consider only case~4.b) (case 4.a) is easier and is treated similarly).

%By considering $f^T$ instead of $f$ for a sufficiently large $T$, we may assume that $C=1$ (the constant from the formulation of the theorem). Clearly the shadowing for $f^T$ implies the corresponding shadowing for $f$.

We will define the coordinates in some fixed small neighborhood of the boundary $U(\epsilon_0,\partial M)$. %For any point $x\in\partial M$ define the map $g_x$ in some neighborhood $N(\delta_0,x)$ such that the border is straightened and mapped to the hypersurface and the neighborhood is mapped to the half-cube. For $x\in M$ the map $g_x$ defined as shift of the corresponding $g_{\bar{x}}$ with $\bar{x}\in \partial M$ ($\bar{x}$ is the closest point of $\partial M$ to $x$). By compactness there exists a constant $C$ such that
%$$|Dg_x|\leq C,\quad |Dg^{-1}_x|\leq C,\qquad \forall x\in M.$$
%We represent a point $p$ in coordinates centered at $x$ as $(y(p)^{(1)},\ldots, y(p)^{(s+u)},)$, where $y(p)(s) = (y(p)^{(1)},\ldots, y(p)^{(s)})$ correspond to the stable space coordinates and $y(p)(u) = (y(p)^{(s+1)},\ldots, y(p)^{(u)})$ correspond to the unstable space coordinates, and $y(p)^{(1)}$ corresponds to the direction orthogonal to the border.(we assume that $\delta_0$ is small, hence, it is possible to choose the stable and unstable coordinates and the transversal to the border coordinate uniformly in the coordinate system centered at $x$).
We introduce a finite (but large number) of coordinate charts. We denote by $\epsilon_1$ maximal diameter of the charts.  In any of the coordinate charts the boundary is mapped to a hyperplane. This hyperplane contains center of the coordinate chart. We denote by $\theta(x)$ the coordinate of $x\in M$ in one of the charts. Any chart is the set of points $\{|\theta(x)|\leq 1\}$.
We represent $\theta(x)$ as $(\theta(x)^{(1)},\ldots, \theta(x)^{(s+u)})$, where $\theta(x)^{\{s\}} = (\theta(x)^{(1)},\ldots, \theta(x)^{(s)})$ correspond to the stable space coordinates and $\theta(x)^{\{u\}} = (\theta(x)^{(s+1)},\ldots, \theta(x)^{(u)})$ correspond to the unstable space coordinates, and $\theta(x)^{(1)}$ corresponds to the direction orthogonal to the boundary (we assume that $\epsilon_1$ is small, hence, it is possible to choose the stable and unstable coordinates and the transversal to the boundary coordinate uniformly in the coordinate system centered at $x$).
Denote by $C$ the Lipschitz constant of the coordinate maps $\theta$ and their inverses $\theta^{-1}$. %Since $C$ can be chosen arbitrarily close to $1$ we assume that it satisfies
Note that since  we can assume that $C$ is sufficiently close to $1$,
\begin{equation}
\label{Crestr}
C^{2m}\mu_2^m<\lambda_{min}^s
\end{equation}
in case 1) and similar inequalities in other cases.
%We assume that $C$ is sufficiently close to $1$.
 Besides, we assume that the coordinates are chosen such that if two points $z_1$ and $z_2$ are contained in one chart, then
$$g(\theta(z_2)) = g(\theta(z_1)) + A(z_1)(\theta(z_2) - \theta(z_1)) + \phi(z_1)(\theta(z_2)-\theta(z_1)),\quad \phi(z_1)(z) = o(z)$$
(where $g$ describes the dynamics in the coordinates).
%%%%the problem is that in this coordinates $g(0)\neq 0$.

In particular, the coordinates are constructed in such a way that the transversal to the boundary tangent direction is orthogonal to the boundary.
Besides, we assume that the coordinate charts are monotonous in the following sense. Suppose that two points $z_1$ and $z_2$ are contained in an image of some coordinate chart; then
$$\mbox{dist}(z_1,\partial M) < \mbox{dist}(z_2,\partial M) \quad\Longleftrightarrow\quad
\theta(z_1)^{(1)} < \theta(z_2)^{(1)}.$$

Let $\{x_k\}_{k\geq0}$ be a nonuniform pseudotrajectory with $n\geq 1$. 
Suppose that a point $p$ is $\epsilon_1$-close to $x_k$ and a point $f(p)$ is $\epsilon_1$-close to $f(x_k)$; then
using Main assumption we conclude that in the coordinates that contain both $f(p)$~and~$f(x_k)$

$$\theta(f(x_k))^{(1)} = \theta(f(p))^{(1)} + \mu(x_k) \theta(p)^{(1)} + \phi(x_k)^{1}(\theta(p)),$$
\begin{equation}
\label{coordrelSt}
\theta(f(x_k))^{\{s\}} = \theta(f(p))^{(s)} + A(x_k)^{s}\theta(p)^{\{s\}} + \phi(x_k)^{s}(\theta(p)),
\end{equation}
\begin{equation}
\label{coordrelUnst}
\theta(f(x_k))^{\{u\}} = \theta(f(p))^{(u)} + A(x_k)^{u}\theta(p)^{\{u\}} + \phi(x_k)^{u}(\theta(p)).
\end{equation}
Assumptions of the theorem imply the corresponding estimates on the products of $|\mu(x_k)|$, $|A(x_k)^s|$, and $|A(x_k)^u|$.

By invariance of the direction towards the boundary and of the stable and unstable spaces,
 $\phi(x_k)^1(\theta(p)) = o(\theta(p)^{(1)})$, $\phi(x_k)^{s}(\theta(p)) = o(\theta(p)^{\{s\}})$, $\phi(x_k)^{u}(y_k) = o(\theta(p)^{\{u\}})$.

Let $\Delta_0$ be a sufficiently small number (if necessary we decrease $\epsilon_1$) such that locally in each chart for $\phi(x_k)(\theta(p))$ we have
\begin{equation}
\label{phi1}
|\phi(x_k)(\theta(p))|\leq\Delta_0|\theta(p)|,\quad|\phi(x_k)^1(\theta(p))|\leq \Delta_0|\theta(p)^{(1)}|,%,\quad|\phi_k^{\textbf{n}}(x)|\leq \delta_0|x^{(\textbf{n})}|,
\end{equation}
\begin{equation}
\label{phi2}
|\phi(x_k)^{s}(\theta(p))|\leq\Delta_0|\theta(p)^{\{s\}}|,\quad |\phi(x_k)^{u}(\theta(p))|\leq\Delta_0|\theta(p)^{\{u\}}|.
\end{equation}

Note that these formulas imply monotonicity of sufficiently precise nonuniform pseudotrajectories with respect to the boundary (for $n\geq 1$). We consider only the case of $\mu_2<1$ (the case of $\mu_1>1$ is similar).
Observe that 
$$|\theta(x_{k+1})^{(1)}|\leq \mu_2 |\theta(x_k)^{(1)}| + o(|\theta(x_k)^{(1)}|) + d((f(x_k))^{(1)}C) \leq$$ 
$$\leq\mu_2 |\theta(x_k)^{(1)}| + o(|\theta(x_k)^{(1)}|) + \delta(\mu_2 C |\theta(x_k)^{(1)}| + o(|\theta(x_k)^{(1)}|))^{n},$$
$$|\theta(x_{k+1})^{(1)}|\geq \mu_1 |\theta(x_k)^{(1)}| + o(|\theta(x_k)^{(1)}|) - \delta(\mu_1 C |\theta(x_k)^{(1)}| + o(|\theta(x_k)^{(1)}|))^{n}$$
(we suppose that $d(\cdot)$ satisfies \eqref{ddef}). 
Moreover, (since, by \eqref{Crestr}, $C\mu_2<1$, and $\delta$ can be chosen so small that $2\delta\mu_2 C\leq \Delta_0$) it follows from $(\ref{phi1})$ that
\begin{equation}
\label{restim}
(\mu_1-2\Delta_0) |\theta(x_k)^{(1)}| \leq |\theta(x_{k+1})^{(1)}| \leq (\mu_2+2\Delta_0) |\theta(x_k)^{(1)}|. 
\end{equation}
By decreasing (if necessary) $\epsilon_1$ (and, consequently, $\Delta_0$ too), we assume that $$\Delta_0 < \min(|1-\mu_1|/4,|1-\mu_2|/4).$$

By the choice of coordinate charts,
estimates $(\ref{restim})$ imply that
$$\mbox{dist}(x_{k+1},\partial B_N) < \mbox{dist}(x_k,\partial B_N)$$
if $\mu_2<1$ and  
$$\mbox{dist}(x_{k+1},\partial B_N) > \mbox{dist}(x_k,\partial B_N)$$
if $\mu_1>1$.

Note that we got monotonicity for $n\geq 1$.
Generally speaking, for $0<n<1$ monotonicity does not hold.
%If $g_1(\cdots)$ were $o(x_1)$, then we would have monotonicity.
%However in general case we have $g(x) = o(x)$ not $g(x) = o(x_1)$.
%May be it is possible to prove monotonicity using other methods
%because we do not have any counterexamples at the moment.
%For linear case monotonicity is obvious; exact condition is the following:
%$$x_k > \mu x_k + d(r(\mu x_k)),$$
%$$d(r(\mu x_k)) < (1-\mu)x_k.$$
%Suppose that $d(r(\mu x_k)) = c r(\mu x_k)^s$.
%Then we easily get monotonicity for $s\geq 1$.
%But for $s<1$ monotonicity is not true, in general!
That is why we require $n\geq1$ even if $0<m<1$.

Hereinafter, we assume that $d(\cdot)\leq \epsilon(\cdot)/L$ for sufficiently large $L$ (if $m\geq1$ it is sufficient to take $d(\cdot)=\epsilon(\cdot)/L$).

%%%%%âîçìîæíî, íàì óäàñòñÿ îòáðîñèòü ýòîò êîñòûëü, ñâÿçàííûé ñ ìîíîòîííîñòüþ, â áóäóùåì

Without loss of generality we assume that $\delta$ is sufficiently small such that
for any two points $q_1$ and $q_2$ (in one of the coordinate charts) that are $\delta$-close
\begin{equation}
\label{zdef1}
d_H(pr_u\chi_{q_1}U(q_1,\epsilon_1),pr_u\chi_{q_2}U(q_2,\epsilon_1))\leq 1/L,
\end{equation}
\begin{equation}
\label{zdef2}
d_H(pr_s\chi_{q_1}U(q_1,\epsilon_1),pr_s\chi_{q_2}U(q_2,\epsilon_1))\leq 1/L,
\end{equation}
where $\chi_{q}$ is the analog of the map $\chi_{q}$ defined at the beginning of Section 4.2.

\textbf{Step 2. The method of Conley.}

Let $\{x_k\}$ be a sufficiently precise nonuniform pseudotrajectory contained in $U(\epsilon_1,\partial M)$. %We assume that it does not intersect the border (since otherwise it is trivial).
%By $U(x,r)$ denote a small ball of size $r$ with center at $x$ (in the corresponding coordinate chart).
We use notations from Section 4.2.
%Put $\delta(x_k) = \mbox{dist}(x_k,\partial M)/2$,  and let $W_{\delta(x_k)}(x_k)$ be the corresponding neighborhood of $x_k$. Without of loss of generality we assume the analogs of $\eqref{Udef}$ with $U_k = U(x_k,\epsilon(\theta(x_k)^{(1)} C))$.
Put $\delta(x_k) = \epsilon(\theta(x_k)^{(1)} C)$. Without of loss of generality we assume the analogs of $\eqref{Udef}$.

%We impose the following conditions on $d(y) = cy^t$:\\
%1) $t\geq1$,
%2) $d(y)\leq\Delta\epsilon(y)$, where $\Delta$ is small.

\textbf{Proof of Item 1) (Case 1).} Suppose that inequality $(\ref{muineq})$ holds. Without loss of generality by \eqref{Crestr}, assume that $L$ is so large that it satisfies the following inequality
\begin{equation}
\label{ineqL}
%%%L>\frac{\mu}{(s^{-}_{min}-\mu^m)^{1/m}}.
C^{2m}(1+1/L)(1+\Delta/L^m)\mu_2^m < \lambda^{s}_{min}.
\end{equation} 
%By min-max principle for singular values,
%$$s^{-}_{min} = \min_{|v|=1}|A^{s}v|.$$
%Of course, since $\det A^{s}\neq0$, $s^{-}_{min}>0$.

The dynamics of $f$ in Case 1 is depicted on Fig. \ref{Fig2}.  

\begin{figure}[h]
	\centering
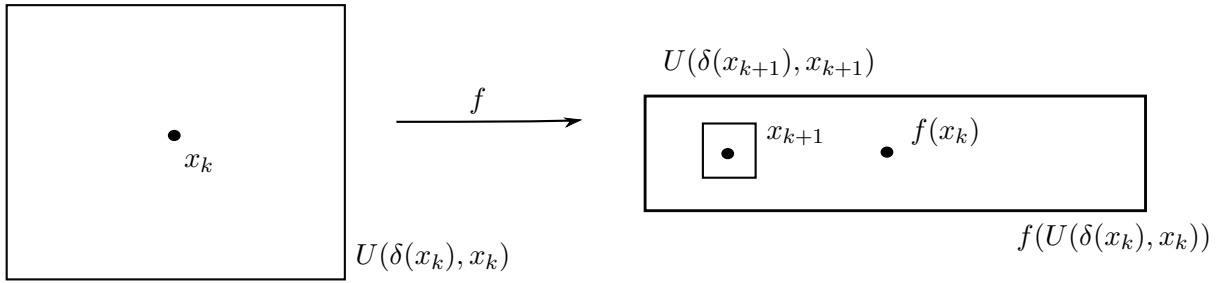
\caption{Mappings of squares in Case 1.}
\label{Fig2}
\end{figure}

%%%%%In essense this case means that all my tangent space is something like quasi-unstable space
%%%%%That is why we have no problems with nonlinearities
We need to prove the following inequality
\begin{equation}
\label{claimup}
U(\theta(x_{k+1}),\epsilon(|\theta(x_{k+1})^{(1)}|C)\subset g_k(U(\theta(x_k),\epsilon(|\theta(x_k)^{(1)}|/C)))
\end{equation}
(where $g_k$ is the corresponding coordinate representation of $f$).
Let us remind the reader that $\epsilon(\cdot)$ satisfies $\eqref{epsdef}$.
Since the Hausdorff distance between $U(\theta(x_{k+1}),\epsilon(|\theta(x_{k+1})^{(1)}|C))$ and $U(\theta(x_{k+1}),\epsilon(|\theta(f(x_k))^{(1)}|C))$
is no more than $\Delta(d(|\theta(f(x_k))^{(1)}|))^m C^{m}\leq \Delta(\epsilon(|\theta(f(x_k))^{(1)}|)/L)^m C^m\leq \Delta C^m/(L^m)$, inequality $(\ref{claimup})$ would follow from
\begin{equation}
\label{claim}
U(\theta(x_{k+1}),(1+\Delta/L^m)\epsilon(|\theta(f(x_k))^{(1)}|)C^m)\subset g_k(U(\theta(x_k),\epsilon(|\theta(x_k)|^{(1)})/C^m)).
\end{equation}

%%%%%%%%%Check it!!!!
Note that for any $r>0$
\begin{equation}
\label{dpstincl}
U(\theta(x_{k+1}),r) \subset U(\theta(f(x_k)),(1+1/L)r).
\end{equation}
It follows from $(\ref{dpstincl})$ that 
$$U(\theta(x_{k+1}),(1+\Delta/L^m)\epsilon(|\theta(f(x_k))^{(1)}|)C^m)\subset$$
$$\subset U(\theta(f(x_k)),(1+1/L)(1+\Delta/L^m)\epsilon(|\theta(f(x_k))^{(1)}|)C^m).$$ 
Note that (since we may assume that $\epsilon_1$ (and, hence, $\Delta_0$) is sufficiently small)
$$U(\theta(f(x_k)),(\lambda^{s}_{min} - \Delta_0)\epsilon(|\theta(x_k)^{(1)}|)/C^m)\subset g_k(U(\theta(x_k),\epsilon(|\theta(x_k)^{(1)}|)/C^m)).$$
Thus in order to get $(\ref{claim})$ it is enough to prove the following inequality
\begin{equation}
\label{ineq}
(1+1/L)(1+\Delta/L^m)\epsilon(|\theta(f(x_k))^{(1)}|)C^m < (\lambda^{s}_{min} - \Delta_0)\epsilon(|\theta(x_k)^{(1)}|)/C^m.
\end{equation}

%By the choice of $\delta_0$ $$r(f(x_k)) \leq (\mu+\delta_0) |x_k^{(\textbf{n})}|,$$
Inequality $(\ref{ineq})$ would follow from $(\ref{restim})$ and
\begin{equation}
\label{ineq2}
(1+1/L)(1+\Delta/L^m)\delta(\mu_2+\Delta_0)^m(|\theta(x_k)^{(1)}|)^m C^m < \delta(\lambda^{s}_{min}-\Delta_0)(|\theta(x_k)^{(1)}|)^m/C^m
\end{equation}
Note that inequality $(\ref{ineq2})$ for any sufficiently small $\Delta_0$ follows from
$(\ref{muineq})$ (one of conditions of the theorem) and $(\ref{ineqL})$.

Inequality $(\ref{claimup})$ implies the analog of relation $(\ref{unstable})$ (where the \textit{quasi-unstable} space is $\mathbb{R}^N$, i.e. $\textbf{s}$ is 0 and $\textbf{u}$ is changed to $s+u$). 
Thus it follows from Lemma~$\ref{LK}$ (applied to the sequence of maps $\{g_k\}_k$) that there exists a unique point $q$ such that 
$$f^k(q)\in U(x_{k},\epsilon(r(x_{k})))\quad\forall k\geq0.$$
Item 1) is proved.

The case when only a finite part of a nonuniform pseudotrajectory is contained in $U(\epsilon_1,\Lambda)$ is treated similarly.

\textbf{Proof of Item 2) (Case 2).} Suppose that inequality $(\ref{muineq2})$ holds.

Assume that a nonuniform pseudotrajectory $\{x_k\}_{k\geq0}$ is fully contained in $U(\epsilon_1,\partial M)$.
The dynamics of $f$ in Case 2 is depicted on Fig. \ref{Fig3}.

We will establish nonuniform shadowing for the sequence of maps $\{g_k\}$, and then transfer the property to the map $f$ (it will change only the constant but not the exponent).

%%%%

%%%èäåîëîãè÷åñêè ìû òóò ñîâåðøåííî íè÷åãî íå ìåíÿåì, îñòàåìñÿ â òðèõîòîìèè: íîðìàëü, ñæàòèå, ðàñòÿæåíèå
%%%clearly this case just means that we have a standard saddle and dimension of stable and unstable manifolds are not cnanged
%%%no linearization is required

%%%%ìîæåò áûòü íàì òóò ðàññóæäàòü â òåðìèíàõ ãèïåðáîëè÷åñêîé ñòðóêòóðû â îêðåñòíîñòè òî÷êè
%%%%è åå ïîëóèíâàðèàíòíîñòè
%We have a semiinvariant hyperbolic structure in a small neighborhood of the point $p$
%íî â ïðèíöèïå è åå íåäîñòàòî÷íî

We use notations from Section 4.2.
As before, let $pr_s$ and $pr_u$ be natural projections on stable and unstable manifolds along unstable and stable manifolds, respectively, (the stable manifold corresponds to $\theta(x)^{\{u\}} = 0$, and the unstable manifold corresponds to $\theta(x)^{\{s\}} = 0$), i.e. $pr_s x:= x^{\{s\}}$, $pr_u x := x^{\{u\}}$.
Let us check the following analogs of relations $(\ref{stable})$ and $(\ref{unstable})$:
\begin{equation}
\label{claim3a}
pr_s \chi_{\theta(x_{k+1})} U(\theta(x_{k+1}),\epsilon(\theta(x_{k+1})^{(1)}/C^m))\supset pr_s \chi_{\theta(x_{k+1})} (g_k(U(\theta(x_k),\epsilon(\theta(x_k)^{(1)})C^m))),
\end{equation}
\begin{equation}
\label{claim3b} 
pr_u \chi_{\theta(x_{k+1})} U(\theta(x_{k+1}),\epsilon(\theta(x_{k+1})^{(1)}C^m))\subset pr_u \chi_{\theta(x_{k+1})} (g_k(U(\theta(x_k),\epsilon(\theta(x_k)^{(1)})/C^m))).
\end{equation}

Since the Hausdorff distance between projections ($pr_s$ or $pr_u$ respectively) of $U(\theta(x_{k+1}),\epsilon(\theta(x_{k+1})^{(1)}))$  and $U(\theta(x_{k+1}),\epsilon(\theta(f(x_k))^{(1)}))$
is less than $\Delta(d(\theta(f(x_k))^{(1)}))^m\leq\Delta(\epsilon(\theta(f(x_k))^{(1)})/L)^m\leq \Delta/L^m$, it is sufficient to prove that
$$
pr_s \chi_{\theta(x_{k+1})}(U(\theta(x_{k+1}),(1-\Delta/L^m)\epsilon(\theta(f(x_k))^{(1)})/C^m))\supset
$$
\begin{equation} 
\label{claim2a} pr_s \chi_{\theta(x_{k+1})}(g_k(U(\theta(x_k),\epsilon(\theta(x_k)^{(1)})C^m))),
\end{equation}
$$
pr_u \chi_{\theta(x_{k+1})} (U(\theta(x_{k+1}),(1+\Delta/L^m)\epsilon(\theta(f(x_k))^{(1)})C^m))\subset
$$
\begin{equation}
\label{claim2b} 
 pr_u \chi_{\theta(x_{k+1})}(g_k(U(\theta(x_k),\epsilon(\theta(x_k)^{(1)})/C^m))).
\end{equation}

%%%%%%%%%%%%%%%%%%%%%%%%%%%%%%%
%%%%%%to modify
%%%%%%%%%%%%%%%%%%%%%%%%%%%%%%%

First we prove inclusion $(\ref{claim2b})$. 
It follows from $\eqref{zdef1}$ that (since $\Delta$ can be chosen to be sufficiently small) in order to get inclusion $(\ref{claim2b})$ it is sufficient to prove the following inclusion:
$$
(1+1/L)pr_u \chi_{\theta(f(x_k))} U(\theta(f(x_k)),(1+\Delta/L^m)\epsilon(\theta(f(x_k))^{(1)})C^m)\subset
$$ 
\begin{equation}
\label{claim2balt}
(1-1/L)pr_u \chi_{\theta(f(x_k))}(g_k(U(\theta(x_k),\epsilon(\theta(x_k)^{(1)})/C^m))).
\end{equation}

Fix $z\in U(\theta(x_k),\epsilon(\theta(x_k)^{(1)}))$.
By $\eqref{coordrelUnst}$ applied to $x_k$ and $z$, and by $\eqref{phi2}$, 
$$|pr_u \theta(f(x_k)) - pr_u \theta(f(z))| \geq (\lambda^{u}_{min} - \Delta_0)|\theta(x_k)^{\{u\}} - \theta(z)^{\{u\}}|.$$
Thus, in order to obtain \eqref{claim2balt}, it is sufficient to prove that
$$(1+1/L)(1+\Delta/L^m)\epsilon(\theta(f(x_k))^{(1)})C^m \leq (1-1/L)(\lambda^{u}_{min} - \Delta_0)\epsilon(\theta(x_k)^{(1)})/C^m.$$
We obtain this inequality as soon as we prove that
$$(1+1/L)(1+\Delta/L^m)(\mu_2+\Delta_0)^m |\theta(x_k)^{(1)}|^m C^m\leq (1-1/L)(\lambda^{u}_{min} - \Delta_0)|\theta(x_k)^{(1)}|^m/C^m.$$
However the last inequality holds trivially for any sufficiently small $\Delta_0$, $\Delta$, any sufficiently large $L$, and any $C$ that is sufficiently close to $1$, since $\mu_2<1$ and $\lambda^{u}_{min}>1$. Inclusion \eqref{claim2b} (and, hence, inclusion \eqref{claim3b}) is proved.

Let us prove inclusion $(\ref{claim2a})$.

It follows from \eqref{zdef2} that, in order to get inclusion \eqref{claim2a}, it is sufficient to obtain the following inclusion:
$$
(1-1/L)pr_{s}\chi_{\theta(f(x_k))}U(\theta(f(x_k)), (1-\Delta/L^m)\epsilon(\theta(f(x_k))^{(1)})/C^m)
\supset
$$
\begin{equation}
\label{claim2aalt}
(1+1/L)pr_{s}\chi_{\theta(f(x_k))}(g_k(U(\theta(x_k),\epsilon(\theta(x_k)^{(1)})C^m))).
\end{equation}
Fix $z\in U(\theta(x_k), \epsilon((\theta(x_k))^{(1)}))$. By \eqref{coordrelSt} applied to $x_k$ and $z$, and by \eqref{phi2},
$$|pr_{s}\theta(f(x_k)) - pr_{s}\theta(f(z))|\leq (\lambda^{s}_{max} + \Delta_0)|\theta(x_k)^{\{s\}} - \theta(z)^{\{s\}}|.$$
Thus, in order to prove \eqref{claim2aalt}, it is sufficient to get the inequality
$$(1-1/L)(1-\Delta/L^m)\epsilon(\theta(f(x_k))^{(1)})/C^m \geq (1+1/L)(\lambda^{s}_{max} + \Delta_0)\epsilon(\theta(x_k)^{(1)})C^m.$$
We will prove this inequality as soon as we prove that
$$(1-1/L)(1-\Delta/L^m)(\mu_1-\Delta_0)^m|\theta(x_k)^{(1)}|^{m}/C^m\geq (1+1/L)(\lambda^{s}_{max} + \Delta_0)|\theta(x_k)^{(1)}|^{m}C^m.$$
However it follows from \eqref{muineq2} (one of the conditions of the theorem) that 
the last inequality holds for any sufficiently small $\Delta_0$, $\Delta$, any sufficiently large $L$, and any $C$ that is sufficiently close to $1$.
Inclusion \eqref{claim2a} (and, consequently, inclusion~\eqref{claim3a}) is proved.

Relations \eqref{claim3a} and \eqref{claim3b} are analogs of relations
$(\ref{stable})$ and $(\ref{unstable})$. Thus we can apply Lemma \ref{LK}.
It follows from Lemma \ref{LK} that there exists an $s$-dimensional disk $D_s$ such that for any point $q\in D_s$ 
$$f^k(q)\in U(x_{k},\epsilon(r(x_{k}))\quad\forall k\geq0.$$

The case when only finite part of a pseudotrajectory is contained in $U(\epsilon_1,\Lambda)$ can be treated similarly.

%\textbf{Proof of Items 2.1)--2.2).} The reasoning used above with minor changes can be used to prove Items 2.1)--2.2). 

\end{proof}

\subsection{Nonuniform shadowing for flows.}

In these section we formulate the analogs of Theorem \ref{ThM} for flows.

Let $\Phi:\mathbb{R}\times M\mapsto M$ be a flow on a smooth compact Riemannian manifold~$M$ with boundary. Assume that $M$ is embedded in an Euclidean space of sufficiently large dimension. 
We assume that $\Phi$ satisfies the following:

\textbf{Main Assumption for flows.}
There exists a compact locally maximal invariant set $\Lambda\subset \partial M$ such that
for any point $p\in\Lambda$ there exists a one-dimensional subspace $\ell(p)\subset T_{p}M$ such that for some sufficiently large constant $C_0$

0) $\ell(p)\notin T_{p}\partial M$;

1) $\ell(\Phi(t,p)) = D\Phi(t,p)\ell(p)$ for any $t\in\mathbb{R}$; 

2) $\ell(p)$ continuously depends on $p$; 

3) $(\exp(\mu_1 t)/C_0)|v|\leq |D\Phi(t,p)v|\leq C_0\exp(\mu_2 t)|v|$ for any $v\in\ell(p)$, $t\in\mathbb{R}$;

4.a) (if $\mu_2<0$) we choose $\lambda^{s}_{min}$ such that $|D\Phi(t,p)v|\geq\exp(\lambda^{s}_{min}t)|v|/C_{0}$ for any $v\in T_{p}M$, $t\in\mathbb{R}$

OR

4.b) $\Lambda$ is a hyperbolic set for $\Phi$ in the sense of \cite{Pil}, i.e. there exist invariant subspaces $S(p), U(p)\subset T_{p}M$ such that $S(p)\oplus U(p) = T_{p}M$ if $p$ is a fixed point of $\Phi$ or $S(p)\oplus U(p)$ has codimension 1 in $T_{p}M$ (and is transversal to the vector field) if $p$ is not a fixed point, and there exist numbers $\lambda^{s}_{min}\leq\lambda^{s}_{max}<0$, $0<\lambda^{u}_{min}\leq\lambda^{u}_{max}$ such that
$$e^{\lambda^{s}_{min}t}|v|/C_0\leq |D\Phi(t,p)v|\leq C_0 e^{\lambda^{s}_{max}t}|v|,\quad\forall v\in S(p), t\in\mathbb{R},$$
$$e^{\lambda^{u}_{min}t}|v|/C_0\leq |D\Phi(t,p)v|\leq C_0 e^{\lambda^{u}_{max}t}|v|,\quad\forall v\in U(p), t\in\mathbb{R}.$$

For simplicity we assume that $\mu_2<0$ (case $\mu_1>0$ is similar).

%We claim that the analogs of Theorems \ref{ThM} (\ref{ThMalt}) and \ref{Th2} hold for flows on $B_N$. Indeed, it is enough to follow the following plan:
%
%1) Any $(m,\delta,1)$-nonuniform pseudotrajectory of a flow is a $(m,\delta)$-pseudotrajectory for the time-one map.
%
%2) By one of the theorems a $(m,\delta)$-pseudotrajectory is nonuniformly shadowed by an exact trajectory of the point $q$ of a time-one map.
%
%3) ote that $$|\Phi(t,q) - \Psi(t)|\leq C\epsilon(r(\Phi(t,p))),$$
%where $C=\max_{x\in\partial B_n,|\tau|\leq 1} D\Phi(\tau,x)$.
%
%Now let us explain how to apply the obtained theorems in order to study polynomial gradient systems. It follows from $(\ref{transfer})$ and $(\ref{CP})$ that if for some pseudotrajectory $\Psi(t)$ and exact trajectory $\Phi(t,p)$ in a small neighborhood of the border we have
%$$|CP(\Psi(t)) - CP(\Phi(t,p))|\leq Cr(CP(\Phi(t,p)))^m,$$
%then
%$$|\Psi(t) - \Phi(t,p)|\leq C|\Phi(t,p)|^{m-2}$$
%(since clearly $|\Phi(t,p)|r(\Phi(t,p))=1$).

Theorem \ref{ThM} can be generalized for flows on $M$ in the following way (for simplicity, we treat only the case $\mu_2<0$, and naturally the corresponding analogs for finite shadowing also hold):
\begin{Theorem}
\label{Th2}
Let $\Phi:\mathbb{R}\times M\mapsto M$ be a flow, and let $\Lambda\subset \partial M$ be a compact locally maximal invariant
set. Let $U$ be a sufficiently small neighborhood of $\Lambda$ (such that the analogs of estimates from Main Assumption hold in $U$).

1) Suppose that Main Assumption for flows with Item 4.a) holds, and $m$ is a sufficiently large number such that
\begin{equation}
\label{muineqflows}
e^{\mu_2 m}<e^{\lambda_{max}^{s}},\quad m>\lambda^{s}_{max}/\mu_2;
\end{equation}
then $\Phi$ has nonuniform shadowing with exponent $m$ by a unique point.

2) Suppose that Main Assumption for flows with Item 4.b) holds, and $m$ is such that
\begin{equation}
\label{muineqflows2}
e^{\mu_1 m}>e^{\lambda_{min}^{s}},\quad m<\lambda^{s}_{min}/\mu_1;
\end{equation}
then $\Phi$ has nonuniform shadowing with exponent $m$.
\end{Theorem}

\begin{Remark}
Note that, since $\mu_2\leq\lambda^s_{max}$, in Item 1) $m>1$, and, since $\mu_1\geq\lambda^{s}_{min}$, in Item 2) $0<m<1$.
\end{Remark}

\begin{proof}
We start from Item 1). Consider the time-$T$ map $f$ for $\Phi$, where $T$ is a sufficiently large number (such that $(\lambda^{s}_{max})^{T}C<1$, $(\lambda^{u}_{min})^{T}/C>1$). Note that it satisfies Main Assumption for diffeomorphisms with Item 4.a) (clearly~\eqref{muineqflows} implies \eqref{muineq}). Thus by Theorem \ref{ThM} $f$ has nonuniform shadowing with exponent~$m$ (by a unique point). Next we apply Proposition \ref{PropA} and observe that $\Phi$ has nonuniform shadowing with exponent $m$
(by a unique point).

Item 2) is much more technical. That is why in this case we give just a brief outline of the proof. Using reasoning required to prove the standard shadowing lemma for flows (see \cite{Pil}), we conclude that in order to prove nonuniform shadowing for $\Phi$ it is sufficient to prove the analog of Item 2) of Theorem \ref{ThM} for a sequence of diffeomorphisms. This sequence of diffeomorphisms satisfies the analog of Main Assumption for diffeomorphisms with Item 4.b) (the analog of \eqref{muineq2} follows from \eqref{muineqflows2}). The shadowing lemma for a sequence of diffeomorphisms can be proved similarly with Item 2) of Theorem \ref{ThM}, but is much more technical. That is why we do not give a detailed proof here.
\end{proof}

\begin{Theorem}
\label{Th3}
Let $\Phi:\mathbb{R}\times \mathbb{R}^{N}\mapsto \mathbb{R}^{N}$ be a flow and $U$ be a sufficiently small neighborhood of infinity. %Suppose that we know that after the compactification procedure the compactified flow is a Morse flow in the neighborhood $CP(U)$ of the border $\partial B_n$. 

1) Suppose that after the compactification procedure (which includes applying a time change) the compactified flow $\bar{\Phi}:\mathbb{R}\times B_{N}\mapsto B_{N}$ satisfies Main Assumption for flows with Item 4.a), $\mu_2<0$ and $\Lambda = \partial B_N$. Let $m$ be such that \eqref{muineqflows} holds. Then $\Phi$ has noncompact oriented nonuniform shadowing from Proposition \ref{PropB} in $U$ with exponent $3-2m$.

2) Suppose that after the compactification procedure the compactified flow $\bar{\Phi}:\mathbb{R}\times B_{N}\mapsto B_{N}$ satisfies Main Assumption for flows with Item 4.b), $\mu_2<0$ and $\Lambda = \partial B_N$. Let $m$ be such that \eqref{muineqflows2} holds. Then $\Phi$ has noncompact oriented nonuniform shadowing from Proposition \ref{PropB} in $U$ with exponent $3-2m$.
\end{Theorem}

\begin{proof}
The theorem follows from Theorem \ref{Th2} and Proposition \ref{PropB}.
\end{proof}

\begin{Remark}
Note that if $3-2m>0$, i.e. $0<m<3/2$, we get a shadowing property with errors that are not bounded, but their growth is controlled.
Whereas for $m\geq 3/2$ the errors are bounded.
%?) Clearly in Item 1) if $m$ is sufficiently large (i.e. $m>2$), then we have something stronger than standard shadowing.
%2) Naturally, it is possible to formulate Theorems \ref{ThM} and \ref{Th2} for the case of flows.
\end{Remark}

%\textit{Another approach is possible based on reparametrized shadowing for discrete systems. We discretize the given system on $\mathbb{R}^n$, and then study reparametrized shadowing for this system (since it is natural to apply time changes during the compactification procedure even for discrete systems).}

%%%%%%y^{noncomp}/r^2 \leq \epsilon
%%%%%%y^{comp}\leq \epsilon r^2 (where r is something like distance to the border)
%%%%%%So we have m \geq 2, for this m we have shadowing by a unique trajectory (probably) 
%%%%%%We can measure our mistakes, when we go back. Thus even for m<1 it is something interesting.

%%%%%%%%%%%%%%%%%%%%%

\section{Weighted shadowing.}

\subsection{Weighted shadowing for flows on compact manifolds.}

%Let $\Psi(t)$ be a $(d,1)$-pseudotrajectory for a flow $\Phi$, i.e.
%$$|\Psi(t+\tau) - \Phi(\tau,\Psi(t))|\leq d,\quad\forall|\tau|\leq 1.$$
%As before put $\psi(t):=\sup_{|\tau|\leq 1}|\Psi(t+\tau) - \Phi(\tau,\Psi(t))|$.
We use notations from Section 3. As before $$\psi(t):=\sup_{|\tau|\leq 1}|\Psi(t+\tau) - \Phi(\tau,\Psi(t))|.$$ Let us formulate the Theorem about weighted shadowing for flows.

\begin{Theorem}
\label{Th4}
% Let $X(x)$ be a smooth vector field on $B_N$  ($N$-dimensional ball, of course, instead one could write here arbitrary smooth manifold) corresponding to a flow $\Phi$. 
Let $\Phi$ be a flow on a compact smooth Riemannian manifold $M$ with the boundary $\partial M$ (e.g., $M=B_{N}$).
%Suppose that $\frac{dX(x)}{dx}\leq c$ in some open set $U$. 
Let $U$ be a small neighborhood of~$\partial M$.
There exist constants $C>1$ and $L$ such that for any sufficiently small number $d$ and any $(d,1)$-pseudotrajectory such that
\begin{equation}
\label{dpst1}
\int_{t\geq0} C^t\psi(t)dt\leq d
\end{equation}
there exists a point $p$ such that
\begin{equation}
\label{epst1}
\int_{t\geq0} C^t|\Phi(t,p) - \Psi(t)|dt\leq Ld.
\end{equation}
\end{Theorem}

\begin{Remark}
The analog of Theorem \ref{Th4} for discrete dynamical systems was formulated and proved in the book \cite{Pil}. 
\end{Remark}

%Then we prove this theorems (proofs are quite short). And we need a remark that says the following: if our errors converge to zero sufficiently fast, then when we return back we have nice shadowing. What we need in reality is a remark for compactified system that exponential shadowing implies modified shadowing with $m\geq2$, which is clear because we can take a coefficient of the exponent that is larger than our derivative.

%%%\subsection{Proofs.}

%%%The proofs of Theorems \ref{Th4} and \ref{Th5} are quite similar.
%%%Let us show how to prove Theorem \ref{Th4}. 

%Let $\dot(x)=F(x)$ be the autonomous system corresponding to the flow $\Phi$. 
\begin{proof}
%Consider the variational system: $\dot{y} = (\partial X(x)/\partial x)y$. The solution of a Cauchy problem $y(0)=x_0$ for this equation is $(\partial \Phi(\cdot,x_0)/\partial x)x_0$. Thus 
Choose 
\begin{equation}
\label{reduc}
%D\Phi(1,\cdot)\leq \exp(C).
C\geq\max_{|\tau|\leq 1, p\in M} ||D\Phi(\tau, p)||
\end{equation}

Let $\Psi(t)$ be a function that satisfies relations $(\ref{dpst1})$. Put $x_k = \Psi(k)$ for $k\in\mathbb{N}$. Since an integral is a limit of Darboux sums, the sequence $x_k$ satisfies the following discrete analog of $(\ref{dpst1})$
$$\sum C^k|x_{k+1} - \Phi(1,x_k)|\leq d$$
(i.e. it is a weighted pseudotrajectory for $\Phi(1,\cdot)$).
Next due to $(\ref{reduc})$ we apply for $\Phi(1,\cdot)$ the result of Pilyugin for discrete time systems (see \cite{Pil}). Thus there exist a global constant $L$ and a point $q$ such that
$$\sum C^k|x_k - \Phi(k,q)| = \sum C^k|\Psi(k) - \Phi(k,q)|\leq Ld.$$
Let $L_0$ be such that
$|D\Phi(\tau,\cdot)|\leq L_0$ for all $|\tau|\leq 1$.
Since for any $|\tau|\leq 1$
$$|\Psi(k+\tau) - \Phi(k+\tau,q)|\leq |\Psi(k+\tau) - \Phi(\tau,\Psi(k))| + |\Phi(\tau,\Psi(k)) -\Phi(k+\tau,q)|\leq$$ 
$$\leq\psi(t) + C|\Psi(k) - \Phi(k,q)|,$$
and the following holds
$$\sum C^k|\Psi(k+\tau) - \Phi(k+\tau,q)|\leq d + CLd,$$
which gives desired estimate $(\ref{epst1})$ for the integral.

%Theorem \ref{Th5} is proved quite similarly. We consider the map $\Phi(1,\cdot)$, note that in a neighborhood of the set from Theorem we have a $(S_{\lambda},U_{\lambda})$ structure for the time-one map (induced from the corresponding structure for the flow), apply the result of Pilyugin for weighted shadowing (see \cite{Pil}), and translate the obtained result in order to get the relations $(\ref{epst2})$.
\end{proof}

\subsection{Weighted shadowing for flows on $\mathbb{R}^{N}$.}

In this section we formulate a noncompact version of Theorem \ref{Th4}.

%Fix $m>1$.
%Note that weighted shadowing for flows with sufficiently fast increasing weights ($C>\max_{0\leq\tau\leq 1}||D\Phi(\tau,\cdot)||^m$) implies that any weighted pseudotrajectory is nonuniformly close to an exact trajectory with the exponent $m$ (since the errors go to zero faster than the dynamics of the time-one map).

%Thus we have the following result for grow-up:
\begin{Theorem}
\label{Th6}
Let $C$ be a sufficiently large number. There exists a map $\alpha:\mathbb{R}\times\mathbb{R}^N\mapsto\mathbb{R}$ such that if $\Psi(t)$ satisfies the following analog of \eqref{dpst1}
\begin{equation}
\label{dpstlast}
\int_{t\geq 0} C^{5t/2}\psi_{\alpha}(t) dt\leq d,
\end{equation}
where 
$$\psi_{\alpha}(t) = \max_{|\tau|\leq 1} |\Psi(t+\tau) - \Phi(\alpha(\tau,\Psi(t)),\Psi(t))|,$$
 then there exists a point $q$ such that the following analog of $(\ref{epst1})$ holds:
\begin{equation}
\label{epstlast}
\int_{t\geq 0} C^{t}|\Phi(\alpha(t,q),q) - \Psi(t)|dt\leq Ld.
\end{equation}
\end{Theorem}
%If errors go to zero sufficiently fast, then we have weighted shadowing.

\begin{proof}
Denote by $\bar{\Phi}$ the compactified flow. We assume that the number $C$ is such that
$$C\geq\max_{|t|\leq 1,x\in B_N{}}|D\bar{\Phi}(t,x)|.$$
Let $\alpha$ be the inverse to time change used in the compactification.
Note that, by \eqref{transferprime} and \eqref{transferalt}, and since $r(\mathbb{R}_{\geq0})\subset [0,1]$,  
$$|\Theta(\Psi(t+\tau)) - \bar{\Phi}(\tau,\Theta(\Psi(t)))| \leq 
|\Psi(t+\tau) - \Phi(\alpha(\tau,\Psi(t)),\Psi(t))|\cdot$$
$$\cdot|r(\bar{\Phi}(\tau,\Theta(\Psi(t))))^{3/2}|
\leq |\Psi(t+\tau) - \Phi(\alpha(\tau,\Psi(t)),\Psi(t))|.$$
Thus, by \eqref{dpstlast},
$$\int_{t\geq 0}C^{5t/2}\max_{|\tau|\leq 1}|\Theta(\Psi(t+\tau)) - \bar{\Phi}(\tau,\Theta(\Psi(t)))| dt\leq \int_{t\geq 0}C^{5t/2}\psi_{\alpha}(t)dt\leq d,$$
which allows us to apply Theorem \ref{Th4} and (by increasing $C$ even more if necessary) to get for some point $\Theta(q)$ the following analog of \eqref{epst1}
\begin{equation}
\label{epst1mod}
\int_{t\geq 0}C^{5t/2}|\bar{\Phi}(t,\Theta(q)) - \Psi(t)|\leq Ld.
\end{equation}

Similarly, by \eqref{transferprime} and \eqref{transfer},
$$|\Phi(\alpha(t,q),q) - \Psi(t)|\leq \frac{|\bar{\Phi}(t,\Theta(q)) - \Theta(\Psi(t))|}{|r(\bar{\Phi}(t,\Theta(q)))|^{3/2}}\leq |\bar{\Phi}(t,\Theta(q)) - \Theta(\Psi(t))|C^{3t/2}.$$
Thus, we derive \eqref{epstlast} from \eqref{epst1mod}:
$$\int_{t\geq 0} C^{t}|\Phi(\alpha(t,q),q) - \Psi(t)|dt \leq
\int_{t\geq 0} C^{5t/2} |\bar{\Phi}(t,\Theta(q)) - \Theta(\Psi(t))| dt\leq Ld.$$

\end{proof}

\begin{section}{Plans for further research.}

\begin{enumerate}
	\item Analogs of theorems about structural stability and $\Omega$-stability for systems with nonuniform shadowing.
	\item Quantitative study of transfer of nonuniform shadowing via time reparametrizations.
	\item Nonuniform shadowing near nonhyperbolic fixed points (analogs of results from \cite{PP}).
	\item Study of shadowing properties of polynomial ODEs.
\end{enumerate}

\end{section}

\begin{section}{Acknowledgment.}

The work of the author was partially supported by and performed at Centro di ricerca matematica Ennio De Giorgi (Italy). Besides, the work of the author was partially supported by Chebyshev Laboratory (Department of Mathematics and Mechanics, St. Petersburg State University) under RF Goverment grant 11.G34.31.0026, and by Leonhard Euler program.

\end{section}

\end{document}

%% file: Fig1mod.pdf_tex
%% Creator: Inkscape 0.48.5, www.inkscape.org
%% Accompanies image file 'Fig1mod.pdf' (pdf, eps, ps)
\begingroup%
  \makeatletter%
  \providecommand\color[2][]{%
    \errmessage{(Inkscape) Color is used for the text in Inkscape, but the package 'color.sty' is not loaded}%
    \renewcommand\color[2][]{}%
  }%
  \providecommand\transparent[1]{%
    \errmessage{(Inkscape) Transparency is used (non-zero) for the text in Inkscape, but the package 'transparent.sty' is not loaded}%
    \renewcommand\transparent[1]{}%
  }%
  \providecommand\rotatebox[2]{#2}%
  \ifx\svgwidth\undefined%
    \setlength{\unitlength}{323.8bp}%
    \ifx\svgscale\undefined%
      \relax%
    \else%
      \setlength{\unitlength}{\unitlength * \real{\svgscale}}%
    \fi%
  \else%
    \setlength{\unitlength}{\svgwidth}%
  \fi%
  \global\let\svgwidth\undefined%
  \global\let\svgscale\undefined%
  \makeatother%
  \begin{picture}(1,0.39547004)%
    \put(0,0){\includegraphics[width=\unitlength]{Fig1mod.pdf}}%
    \put(0.47,0.03255319){\color[rgb]{0,0,0}\makebox(0,0)[lb]{\smash{$0$}}}%
    \put(0.77573233,0.36256293){\color[rgb]{0,0,0}\makebox(0,0)[lb]{\smash{$x$}}}%
    \put(0.71,0.075){\color[rgb]{0,0,0}\makebox(0,0)[lb]{\smash{$\bar{x}$}}}%
    \put(0.33,0.21){\color[rgb]{0,0,0}\makebox(0,0)[lb]{\smash{$1$}}}%
    \put(0.74,0.00519949){\color[rgb]{0,0,0}\makebox(0,0)[lb]{\smash{$\bar{y}$}}}%
  \end{picture}%
\endgroup%

%% file: Fig2mod.pdf_tex
%% Creator: Inkscape 0.48.5, www.inkscape.org
%% Accompanies image file 'Fig2mod.pdf' (pdf, eps, ps)
\begingroup%
  \makeatletter%
  \providecommand\color[2][]{%
    \errmessage{(Inkscape) Color is used for the text in Inkscape, but the package 'color.sty' is not loaded}%
    \renewcommand\color[2][]{}%
  }%
  \providecommand\transparent[1]{%
    \errmessage{(Inkscape) Transparency is used (non-zero) for the text in Inkscape, but the package 'transparent.sty' is not loaded}%
    \renewcommand\transparent[1]{}%
  }%
  \providecommand\rotatebox[2]{#2}%
  \ifx\svgwidth\undefined%
    \setlength{\unitlength}{492.24384766bp}%
    \ifx\svgscale\undefined%
      \relax%
    \else%
      \setlength{\unitlength}{\unitlength * \real{\svgscale}}%
    \fi%
  \else%
    \setlength{\unitlength}{\svgwidth}%
  \fi%
  \global\let\svgwidth\undefined%
  \global\let\svgscale\undefined%
  \makeatother%
  \begin{picture}(1,0.21319985)%
    \put(0,0){\includegraphics[width=\unitlength]{Fig2mod.pdf}}%
    \put(0.35489816,0.13200389){\color[rgb]{0,0,0}\makebox(0,0)[lb]{\smash{$f$}}}%
    \put(0.13665556,0.08673015){\color[rgb]{0,0,0}\makebox(0,0)[lb]{\smash{$x_k$}}}%
    \put(0.615,0.09){\color[rgb]{0,0,0}\makebox(0,0)[lb]{\smash{$f(x_k)$}}}%
    \put(0.715,0.15289951){\color[rgb]{0,0,0}\makebox(0,0)[lb]{\smash{$x_{k+1}$}}}%
    \put(0.26783328,0.01243489){\color[rgb]{0,0,0}\makebox(0,0)[lb]{\smash{$U(\delta(x_k),x_k)$}}}%
    \put(0.72637488,0.20049495){\color[rgb]{0,0,0}\makebox(0,0)[lb]{\smash{$U(\delta(x_{k+1}),x_{k+1})$}}}%
    \put(0.83897875,0.06931716){\color[rgb]{0,0,0}\makebox(0,0)[lb]{\smash{$f(U(\delta(x_k),x_k))$}}}%
  \end{picture}%
\endgroup%

%% file: Fig3mod.pdf_tex
%% Creator: Inkscape 0.48.5, www.inkscape.org
%% Accompanies image file 'Fig3mod.pdf' (pdf, eps, ps)
\begingroup%
  \makeatletter%
  \providecommand\color[2][]{%
    \errmessage{(Inkscape) Color is used for the text in Inkscape, but the package 'color.sty' is not loaded}%
    \renewcommand\color[2][]{}%
  }%
  \providecommand\transparent[1]{%
    \errmessage{(Inkscape) Transparency is used (non-zero) for the text in Inkscape, but the package 'transparent.sty' is not loaded}%
    \renewcommand\transparent[1]{}%
  }%
  \providecommand\rotatebox[2]{#2}%
  \ifx\svgwidth\undefined%
    \setlength{\unitlength}{460.24384766bp}%
    \ifx\svgscale\undefined%
      \relax%
    \else%
      \setlength{\unitlength}{\unitlength * \real{\svgscale}}%
    \fi%
  \else%
    \setlength{\unitlength}{\svgwidth}%
  \fi%
  \global\let\svgwidth\undefined%
  \global\let\svgscale\undefined%
  \makeatother%
  \begin{picture}(1,0.22645604)%
    \put(0,0){\includegraphics[width=\unitlength]{Fig3mod.pdf}}%
    \put(0.37957365,0.1411819){\color[rgb]{0,0,0}\makebox(0,0)[lb]{\smash{$f$}}}%
    \put(0.146157,0.09276036){\color[rgb]{0,0,0}\makebox(0,0)[lb]{\smash{$x_k$}}}%
    \put(0.74087276,0.11635036){\color[rgb]{0,0,0}\makebox(0,0)[lb]{\smash{$f(x_k)$}}}%
    \put(0.62292291,0.11635041){\color[rgb]{0,0,0}\makebox(0,0)[lb]{\smash{$x_{k+1}$}}}%
    \put(0.28645528,0.01329946){\color[rgb]{0,0,0}\makebox(0,0)[lb]{\smash{$U(\delta(x_k),x_k)$}}}%
    \put(0.53849559,0.17222141){\color[rgb]{0,0,0}\makebox(0,0)[lb]{\smash{$U(\delta(x_{k+1}),x_{k+1})$}}}%
    \put(0.82778321,0.02943985){\color[rgb]{0,0,0}\makebox(0,0)[lb]{\smash{$f(U(\delta(x_k),x_k))$}}}%
  \end{picture}%
\endgroup%